\documentclass[reqno,12pt]{amsart} 

\date{\today}

\usepackage{amsmath}
\usepackage{amsfonts}
\usepackage{amssymb}
\usepackage{amsthm}
\usepackage{amstext}
\usepackage{appendix}

\newcommand{\bx}{\mathbf{x}} 
\newcommand{\bp}{\mathbf{p}}
\newcommand{\N}{\mathbb{N}}
\newcommand{\C}{\mathbb{C}}
\newcommand{\R}{\mathbb{R}}
\newcommand{\supp}{{\operatorname{supp}}}
\newcommand{\dist}{{\operatorname{dist}}}

\theoremstyle{plain}
\newtheorem{theorem}{Theorem}[section]{\bf}{\it}
\newtheorem{prop}[theorem]{Proposition}{\bf}{\it}
\newtheorem{lemma}[theorem]{Lemma}{\bf}{\it}
{\bf}{\it}
{\bf}{\it}
\theoremstyle{definition}
 
\newtheorem{remark}[theorem]{Remark}{\it}{\rm}
{\bf}{\rm}
{\rm}{\rm}

\newcommand{\1}{{\rm 1\hspace*{-0.4ex}%
\rule{0.1ex}{1.52ex}\hspace*{0.2ex}}}

\newenvironment{pf*}[1]{\par\medskip\noindent\textit{#1}\,:}{\hspace*{\fill}\qed\medskip\par\noindent}   

\title[Analyticity of solutions to nonlocal
equations]
{Real analyticity of solutions\\ to Schr{\"o}dinger
  equations\\  
  involving fractional Laplacians}

\author[A. Dall'Acqua, S. Fournais, T. {\O}. S{\o}rensen, and
E. Stockmeyer]{Anna Dall'Acqua, S{\o}ren Fournais,\\ Thomas \O stergaard
  S\o rensen, and Edgardo Stockmeyer} 

\thanks{\copyright\ 2013 by the
       authors. This article may be reproduced, in its entirety, for
       non-commercial purposes.}

\address[Anna Dall'Acqua]
{Max Planck Institute for Mathematics in the Sciences,
Inselstra{\ss}e 22,
D-04103 Leipzig, Germany.}

\address[]{Present address:
Institute of Analysis, University of Ulm, Helmholtz\-strasse 18, D-89091
Ulm, Germany. 
}
\email{anna.dallacqua@uni-ulm.de}

\address[S. Fournais]
{Department of Mathematics,
Aarhus University,
Ny Munke\-gade 118,
DK-8000 Aarhus C, Denmark.}
\email{fournais@imf.au.dk}

\address[Thomas {\O}stergaard S{\o}rensen]
{Department Mathematisches Institut, LMU M{\"u}nchen
Theresienstra\ss e 39, D-80333 Munich, Germany.}
\email{sorensen@math.lmu.de}

\address[Edgardo Stockmeyer]
{Department Mathematisches Institut, LMU M{\"u}nchen\\
Theresienstra\ss e 39, D-80333 Munich, Germany.}
\address[]{Present address:
Instituto de F\'isica, Pontificia Universidad Cat\'olica de Chile,
Vicu\~na Mackenna 4860, Santiago 7820436, Chile.  
}
\email{stock@fis.puc.cl}

\begin{document}

\begin{abstract}
  We prove analyticity of solutions in \(\mathbb{R}^{n}\), \(n\ge1\),
  to certain nonlocal linear Schr{\"o}dinger equations with
  analytic potentials. 
\end{abstract}


\maketitle

\section{Introduction, result, and proof.}\label{sec:first}

In \cite{ours} we proved the real analyticity away from the Coulomb
singularity of atomic pseudorelativistic Hartree-Fock orbitals.  The
proof works for solutions to a variety of equations (see
\cite[Remark~1.2]{ours}), in particular, any \(H^{1/2}\)-solution
\(\varphi:\mathbb{R}^3\to\mathbb{C}\) to the non-linear equation
\begin{align}\label{eq:non-lin1}
  (\sqrt{-\Delta+1})\varphi -
  \frac{Z}{|\cdot|}\varphi\pm\big(|\varphi|^2*|\cdot|^{-1}\big)\varphi
  =\lambda\varphi
\end{align}
is real analytic away from \(\mathbf{x}=0\). The emphasis in
\cite{ours} was on the Coulomb singularity \(|\cdot|^{-1}\), on the
Hartree-term \((|\varphi|^2*|\cdot|^{-1})\varphi\), and on the
dimension 
\(n=3\). However, the 
result holds for much more general potentials \(V\) than
\(|\cdot|^{-1}\), and in any dimension \(n\ge1\). We state and prove
this in the linear case here  
(refering to \cite{ours} for certain technical points of the proof).

\begin{theorem}\label{thm:one}
  Let \(\Omega\subset\mathbb{R}^n\), \(n\ge1\), be an open set, and assume
  \(V:\mathbb{R}^{n}\to\mathbb{C}\) is real analytic in \(\Omega\),
  that is, \(V\in C^{\omega}(\Omega)\).  Let \(s\in[1/2,1)\), \(m>0\),
  or \(s=1/2, m=0\), and assume \(\varphi\in
  {H}^{2s}(\mathbb{R}^{n})\) is a solution to
  \begin{align}\label{eq:main}
    E_{s,m}({\bf p})\varphi:=(-\Delta+m^2)^{s}\varphi = V\varphi \ \
    \text{ in } \ \ L^{2}(\mathbb{R}^{n})\,.
  \end{align}
  Assume furthermore that \( V\in L^{t}(\mathbb{R}^{n}) +
  L^{\infty}(\mathbb{R}^{n})\) with
  \begin{align}\label{cond-L-t-on-V}
    \begin{cases}
      t=n/4s & \text{if $s\in[1/2,n/4)\cap[1/2,1)$ and $n\ge3$\,,}
      \\
      t>1 & \text{if $s=n/4$ and $n=2$ or $n=3$\,,}
      \\
      t=1 & \text{if $s\in(n/4,1)$ and $n=1,2,3$\,.}
    \end{cases}
  \end{align}

  Then \(\varphi\in C^{\omega}(\Omega)\), that is, \(\varphi\) is real
  analytic in \(\Omega\).
\end{theorem}
\begin{remark}\label{rem:first}
  In the case \(s=1\), the result is well-known (for all \(m\ge0\)), and no integrability
  condition on \(V\) is needed, the equation being local in this
  case. The integrability conditions on \(V\) seem unnecessary, but
  are needed for our method to work (see \eqref{the-term} and
  after). Note that if \(V\in L^{p}(\mathbb{R}^{n})\) for some
  \(p\in[1,\infty)\), then \(V\in
  L^{q}(\mathbb{R}^{n})+L^{\infty}(\mathbb{R}^{n})\) for all
  \(q\in[1,p]\).  
As in \cite{ours} our proof is based on the classical proof by Morrey
and Nirenberg (see \cite{Hormander}). In order to deal with the
non-locality we use the localization result in Lemma~\ref{Edgardo} and
the analytic smoothing estimate in Lemma~\ref{normsmooth-Lp}, both in the
Appendix below (for more details see \cite[after Remark~1.4]{ours}).
\end{remark}
To prove Theorem~\ref{thm:one}, it suffices (using Sobolev embedding)
to prove the following proposition (for details in the case \(n=3\), see
\cite[after~Proposition~2.1]{ours}; this can be modified for general
\(n\ge 1\)). Note that in the linear case, it
suffices to work in \(L^2(\mathbb{R}^{n})\).
\begin{prop}\label{prop:first}
  Let the assumptions be as in Theorem~\ref{thm:one}.  Let
  \(\mathbf{x}_{0}\in\Omega\),
  \(R=\min\{1,{\operatorname{dist}}(\mathbf{x}_{0},\Omega^{c})/4\}\),
  and \(\omega=B_{R}(\mathbf{x}_{0})\,(\subset\subset\Omega)\). Define
  \(\omega_{\delta}=B_{R-\delta}(\mathbf{x}_{0})\) for \(\delta>0\).

  Then there exist constants $C,B>1$ such that for all
  $j\in\mathbb{N}_{0}$, and for all $\epsilon >0$ such that \(\epsilon
  j\le R/2\), we have
  \begin{equation}\label{eq:lemma2}
    \epsilon^{|\beta|} \|D^{\beta} \varphi\|_{L^{2}(\omega_{\epsilon j})}
    \leq
    C B^{|\beta|} \qquad \text{for all $\beta \in \mathbb{N}_{0}^n$
      with $|\beta| \leq j\,.$} 
  \end{equation}
\end{prop}
\begin{proof}
  This is by induction. For \(j\in\mathbb{N}_{0}\) (and constants
  \(C,B>1\) to be determined below), let \(\mathcal{P}(j)\) be the
  statement: For all $\epsilon >0$ with \(\epsilon j\le R/2\) we have
  \begin{equation}\label{eq:lemma2-ihy}
    \epsilon^{|\beta|} \|D^{\beta} \varphi\|_{L^{2}(\omega_{\epsilon j})}
    \leq
    C\, B^{|\beta|} \qquad \text{for all $\beta \in \mathbb{N}_{0}^n $
      with $ |\beta|\le j\,.$}
  \end{equation}
  Choosing \(C\ge \|\varphi\|_{H^{1}(\mathbb{R}^{n})}\) (which is
  finite since \(\varphi\in H^{2s}(\mathbb{R}^{n})\), \(s\in[1/2,1)\)) and \(B> 1\)
  ensures that both \(\mathcal{P}(0)\) and \(\mathcal{P}(1)\) hold
(since \(\epsilon\le R/2\le 1\) for \(j=1\)).
  The induction hypothesis is: Let \(j\in\mathbb{N}\), \(j\ge
  1\). Then \(\mathcal{P}(\tilde{j})\) holds for all \(\tilde{j}\le
  j\).  We now prove that \(\mathcal{P}(j+1)\) holds. By the
  definition of \(\omega_{\delta}\) and the induction hypothesis, it
  suffices to study \(\beta\in\mathbb{N}_0^n\) with \(|\beta|=j+1\).
  It therefore remains to prove that
  \begin{align}\label{est:to-prove}\nonumber
    \epsilon^{|\beta|} \|D^{\beta}
    \varphi\|_{L^{2}(\omega_{\epsilon(j+1)})} \leq C\, B^{|\beta|}
    \qquad &\text{for all $\epsilon>0 $ with $\epsilon(j+1)\le R/2$}
    \\
    &\text{and all $ \beta \in \mathbb{N}_{0}^n $ with $
      |\beta|=j+1\,.$}
  \end{align}
  Let \(\epsilon\) and \(\beta\) be as in \eqref{est:to-prove}.  It is
  convenient to write, for \(\ell>0\), \(\epsilon>0\) such that
  \(\epsilon \ell\le R/2\), and \(\sigma\in\mathbb{N}_0^n\) with
  \(0<|\sigma|\le j\),
  \begin{equation*}
    \|D^{\sigma}\varphi\|_{L^{2}(\omega_{\epsilon \ell})}
    = 
    \|D^{\sigma}\varphi\|_{L^{2}(\omega_{\tilde{\epsilon}\tilde{j}})}
    \quad\text{with}\quad
    \tilde{\epsilon}=\frac{\epsilon \ell}{|\sigma|},\ \tilde{j}=|\sigma| 
    \,, 
  \end{equation*}
  so that, by the induction hypothesis (applied on the term with
  $\tilde{\epsilon}$ and \(\tilde{j}\)) we get that
  \begin{align}\label{eq:ihy-new}
    \|D^{\sigma}\varphi\|_{L^{2}(\omega_{\epsilon\ell})} \leq C
    \Big(\frac{B}{\tilde{\epsilon}}\Big)^{|\sigma|} = C \Big(
    \frac{|\sigma|}{\ell} \Big)^{|\sigma|}
    \Big(\frac{B}{\epsilon}\Big)^{|\sigma|}\,.
  \end{align}
  Compare this with \eqref{eq:lemma2-ihy}.  With the convention that
  \(0^0=1\), \eqref{eq:ihy-new} also holds for \(|\sigma|=0\).

  Inverting the equation \eqref{eq:main} when \(m>0\), we have (in
  \(L^{2}(\mathbb{R}^{n})\))
  \begin{align}\label{eq:inverted}
    \varphi = E_{s,m}({\bf p})^{-1}V\varphi\,.
  \end{align}
  For the case \(s=1/2, m=0\),
  \begin{align}\label{eq:inverted-bis}
    \varphi = \big((-\Delta)^{1/2}+1\big)^{-1}\widetilde{V}\varphi
    =:\widetilde{E}_{1/2,0}({\bf p})^{-1}\widetilde{V}\varphi\,,
  \end{align}
  with \(\widetilde{V}=V+1\in
  L^{t}(\mathbb{R}^n)+L^{\infty}(\mathbb{R}^{n})\). Note that \(1\in
  C^{\omega}(\mathbb{R}^{n})\).

  We choose a function $\Phi$ (depending on \(j\)) satisfying
  \begin{equation}\label{def:Phi}
    \Phi \in C^{\infty}_{0}(\omega_{\epsilon (j+3/4)})\,,\quad
    0\le\Phi\le1\,,\quad
    \mbox{
      with }\; \Phi \equiv 1 \; \mbox{ on }\; \omega_{\epsilon(j+1)}\,. 
  \end{equation}
  Then \(\|D^{\beta}\varphi\|_{L^{2}(\omega_{\epsilon(j+1)})} \le
  \|\Phi D^{\beta}\varphi\|_{L^2(\mathbb{R}^{n})}=:\|\Phi
  D^{\beta}\varphi\|_{2}\).  The estimate \eqref{est:to-prove}--and
  hence, by induction, the proof of Proposition~\ref{prop:first}---now
  follows from \eqref{eq:inverted} and \eqref{eq:inverted-bis} and the
  following lemma.
\end{proof}
\begin{lemma} \label{lem:V} Assume the induction hypothesis described
  above holds. Let \(\Phi\) be as in \eqref{def:Phi}.  Then for all
  $\epsilon >0$ with \(\epsilon(j+1)\le R/2\), and all $\beta \in
  \mathbb{N}_{0}^n$ with $|\beta|=j+1$, \(\Phi D^{\beta}
  E_{s,m}(\mathbf{p})^{-1} V \varphi\) belongs to
  \(L^{2}(\mathbb{R}^{n})\), and
  \begin{align}\label{est-V}
    \|\Phi D^{\beta} E_{s,m}(\mathbf{p})^{-1} V \varphi\|_{2} \leq
    C\Big(\frac{B}{\epsilon}\Big)^{|\beta|}\,,
  \end{align}
  where \(C, B>1\) are the constants in \eqref{eq:lemma2-ihy}.

  The  same holds for
  \(\Phi D^{\beta}\widetilde{E}_{1/2,0}({\bf
    p})^{-1}\widetilde{V}\varphi\).
\end{lemma}
\begin{proof}
  Let $\sigma \in \mathbb{N}_{0}^n$ and $\nu\in \{1,\ldots,n\}$ be such
  that $\beta=\sigma +e_{\nu}$, so that
  $D^{\beta}=D_{\nu}D^{\sigma}$. Notice that $|\sigma|=j$.  Choose
  localization functions $\{\chi_{k}\}_{k=0}^{j}$ and
  $\{\eta_{k}\}_{k=0}^{j}$ as described in the Appendix below.  Since
  \(V\varphi\in L^{2}(\mathbb{R}^{n})\) (from \(\varphi\in
  H^{2s}(\mathbb{R}^{n})\) and the equation \eqref{eq:main}), and
  \(E_{s,m}(\mathbf{p})^{-1}\) maps \(H^{r}(\mathbb{R}^{n})\) to
  \(H^{r+2s}(\mathbb{R}^{n})\) for all \(r\in\mathbb{R}\),
  Lemma~\ref{Edgardo} below (with $\ell=j$) implies that
  \begin{align} \label{f2}\notag 
    \Phi  D^{\beta} E_{s,m}(\mathbf{p})^{-1}[V
    \varphi] &=\sum_{k=0}^{j} \Phi
    D_{\nu}E_{s,m}(\mathbf{p})^{-1}D^{\beta_{k}}\chi_{k}
    D^{\sigma-\beta_{k}} [V\varphi] 
    \\\notag &{\ }+ 
   \sum_{k=0}^{j-1} \Phi
    D_{\nu}E_{s,m}(\mathbf{p})^{-1}D^{\beta_{k}}[\eta_{k},D^{\mu_{k}}]
    D^{\sigma-\beta_{k+1}} [V \varphi] 
    \\ &{\ }+ 
    \Phi
    D_{\nu}E_{s,m}(\mathbf{p})^{-1}D^{\sigma}[\eta_{j}V\varphi]\,,
  \end{align}
  as an identity in \(H^{-|\beta|+2s}(\mathbb{R}^{n})\). Similarly for
  \(\widetilde{E}_{1/2,0}({\bf p})^{-1}\).  Here,
  \([\,\cdot\,,\,\cdot\,]\) denotes the commutator. Also,
  \(|\beta_k|=k\), \(|\mu_k|=1\), and \(0\le\eta_k,\chi_k\le1\). (For
  the support properties of \(\eta_k,\chi_k\), see the Appendix.)  We
  will prove that each term on the right side of \eqref{f2} belong to
  \(L^{2}(\mathbb{R}^{n})\), and bound their norms.  The proof of
  \eqref{est-V} will follow by summing these bounds.
  
\ %

\noindent{\bf The first sum in \eqref{f2}.}
  Let $\theta_{k}$ be the characteristic function of the support of
  $\chi_{k}$ (which is contained in \(\omega\)).  We can estimate, for
  \(k\in\{0,\ldots,j\}\),
  \begin{align}\label{f3}\notag
    \|\Phi D_{\nu}E_{s,m}(\mathbf{p})^{-1}&D^{\beta_{k}}\chi_{k}
    D^{\sigma-\beta_{k}} [V \varphi] \|_{2} 
    \\&=\| (\Phi
    E_{s,m}(\mathbf{p})^{-1}D_{\nu}D^{\beta_{k}} \chi_{k}) \theta_k
    D^{\sigma-\beta_{k}} [V \varphi] \|_{2} 
    \notag\\&\leq \| \Phi
    E_{s,m}(\mathbf{p})^{-1}D_{\nu}D^{\beta_{k}} \chi_{k}
    \|_{\mathcal{B}}\, \| \theta_{k} D^{\sigma-\beta_{k}} [ V \varphi]
    \|_{2}\,.
  \end{align}
  Here, \(\|\cdot\|_{\mathcal{B}}\) is the operator norm on the
  bounded operators on \(L^{2}(\mathbb{R}^n)\).

  For \(k=0\), the first factor on the right side of \eqref{f3} can be
  estimated using the Fourier transform, since \(s\in[1/2,1)\).  This
  way, since \(\|\chi_{0}\|_{\infty}=\|\Phi\|_{\infty}=1\),
  \begin{align}\label{eq:est-smooth-factor-k=0}
    \|\Phi E_{s,m}(\mathbf{p})^{-1}D_{\nu}\chi_{0} \|_{\mathcal{B}}
    \le C_{s}(m)\,.
  \end{align}
  This also holds for \(\widetilde{E}_{1/2,0}({\bf p})^{-1}\).

  For \(k>0\), the first factor on the right side of \eqref{f3} can
  (also for \(\widetilde{E}_{1/2,0}({\bf p})^{-1}\)) be estimated
  using \eqref{eq:smoothing-est-Lp} in Lemma~\ref{normsmooth-Lp}
  below (with \(\mathfrak{r}=1\),
  \(\mathfrak{q}=\mathfrak{q}^{*}=\mathfrak{p}=2\); note that
  \(|e_{\nu}+\beta_k|=k+1\ge2\)).  Since  
  \begin{align*}
    {\operatorname{dist}}({\operatorname{supp}}\,\chi_{k},{\operatorname{supp}}\,\Phi)\ge\epsilon
    (k-1+1/4)
  \end{align*}
  and \(\|\chi_k\|_{\infty}=\|\Phi\|_{\infty}=1\), this gives (since
  \((\beta_k+e_{\nu})!\le(|\beta_k|+1)!=(k+1)!\)) that
  \begin{align*}
    \| \Phi E_{s,m}(\mathbf{p})^{-1} &D_{\nu}D^{\beta_{k}} \chi_{k}
    \|_{\mathcal{B}} \\&\le c_{n,s}
    (k+1)!\,\Big(\frac{2n+2}{\epsilon
      (k-1+1/4)}\Big)^{k+1}\big[\epsilon(k-1+1/4)\big]^{2s}\,.
  \end{align*}
  Since \(2s-1\ge0\), and \(\epsilon(k-1+1/4)\le \epsilon(j+1)\le
  R/2\le1\), this implies
 \begin{align}\label{eq:est-smooth-factor-k>0}\notag
    \| \Phi E_{s,m}(\mathbf{p})^{-1} &D_{\nu}D^{\beta_{k}} \chi_{k}
    \|_{\mathcal{B}} \\&\le c_{n,s}8(2n+2)
    \Big(\frac{2n+2}{\epsilon}\Big)^{k}=\tilde{c}_{n,s}
    \Big(\frac{2n+2}{\epsilon}\Big)^{k}\,.
  \end{align}
  It follows from \eqref{eq:est-smooth-factor-k=0} and
  \eqref{eq:est-smooth-factor-k>0} that, for all
  \(k\in\{0,\ldots,j\}\), \(\nu\in\{1,\ldots,n\}\),
  \begin{align}\label{eq:est-smooth-factor-BIS}
    \| \Phi E_{s,m}(\mathbf{p})^{-1} D_{\nu}D^{\beta_k}\chi_{k}
    \|_{\mathcal{B}} \le \widetilde{C}_{n,s}(m)
    \Big(\frac{2n+2}{\epsilon}\Big)^{k} \,,
  \end{align}
with \(\widetilde{C}_{n,s}(m)=\tilde{c}_{n,s}+C_s(m)\).

  It remains to estimate the second factor in \eqref{f3}. For this, we
  employ the analyticity of \(V\).  Let \(A=A(\mathbf{x}_0)\ge1\) be
  such that, for all \(\sigma\in\mathbb{N}_{0}^{n}\),
  \begin{equation}\label{Wanal0-bis}
    \sup_{\mathbf{x} \in \omega} |D^{\sigma} V(\mathbf{x})| \leq A^{|\sigma|+1}
    |\sigma| !\,. 
  \end{equation}
  The existence of \(A\) follows from the real analyticity in
  \(\omega=B_{R}(\mathbf{x}_0)\subset\subset\Omega\) of \(V\) (see
  e.~g.~\cite[Proposition~2.2.10]{Krantz}).  It follows (since
  \(\omega_{\delta}=\emptyset\) for \(\delta\ge 1\)) that, for all
  \(\epsilon>0\), $\ell\in \mathbb{N}_{0}$, and
  \(\sigma\in\mathbb{N}_{0}^{n}\),
  \begin{equation}\label{Wanal}
    \epsilon^{|\sigma|} \sup_{\mathbf{x} \in \omega_{\epsilon \ell}}
    |D^{\sigma} V(\mathbf{x})| \leq A^{|\sigma|+1} |\sigma| ! \;
    {\ell}^{-|\sigma|}\,,    
  \end{equation}
  with \(\omega_{\epsilon \ell}\subseteq\omega\) as defined in
  Proposition~\ref{prop:first}.

  For $k=j$, since $\beta_{j}=\sigma$, we find, by \eqref{Wanal} and
  the choice of \(C\), that
  \begin{equation}\label{f4}
    \| \theta_{j}V\varphi\|_{2} \leq \|V\|_{L^{\infty}(\omega)}
    \|\varphi\|_{L^{2}(\omega)}
    \le CA\,. 
  \end{equation}

  For $k \in \{0, \dots, j-1\}$ we get, by Leibniz's rule, that
  \begin{align}\label{f5}\notag
    \| \theta_{k} D^{\sigma-\beta_{k}} &[V \varphi] \|_{2} \\&\leq
    \sum_{\mu \leq \sigma-\beta_{k}} \binom{\sigma-\beta_k}{\mu} \|
    \theta_{k} D^{\mu} V\|_{\infty} \,
    \|\theta_{k}D^{\sigma-\beta_{k}-\mu} \varphi\|_{2}\,.
  \end{align}
  Now,
  \({\operatorname{supp}}\,\theta_{k}={\operatorname{supp}}\,\chi_{k}\subseteq
  \omega_{\epsilon(j-k+1/4)}\), so by \eqref{Wanal}, for all \(\mu\le
  \sigma-\beta_k\),
  \begin{align}\label{eq:est-infty-norm-V}
    \| \theta_{k} D^{\mu} V\|_{\infty} \le
    \sup_{\mathbf{x}\in\omega_{\epsilon(j-k+1/4)}} |D^{\mu}
    V(\mathbf{x})| \le
    \epsilon^{-|\mu|}A^{|\mu|+1}|\mu|!(j-k)^{-|\mu|}\,.
  \end{align}
  By the induction hypothesis (in \eqref{eq:ihy-new}),
  \begin{align}\label{eq:est-phi-in-V-term}\nonumber
    \|\theta_{k}D^{\sigma-\beta_{k}-\mu} \varphi\|_{2} &\le
    \|D^{\sigma-\beta_{k}-\mu}
    \varphi\|_{L^{2}(\omega_{\epsilon(j-k)})} \\&\le
    C\Big(\frac{|\sigma-\beta_k-\mu|}{j-k}\Big)^{|\sigma-\beta_k-\mu|}
    \Big(\frac{B}{\epsilon}\Big)^{|\sigma-\beta_k-\mu|}\,.
  \end{align}
  It follows from \eqref{f5}, \eqref{eq:est-infty-norm-V}, and \eqref
  {eq:est-phi-in-V-term} (using that \(|\sigma|=j, |\beta_k|=k\), so
  \(\sum_{\mu\le\sigma-\beta_k,
    |\mu|=m}\binom{\sigma-\beta_k}{\mu}=\binom{|\sigma-\beta_k|}{m}=\binom{j-k}{m}\),
  and then summing over \(m\)) that
  \begin{align}\label{eq:first-sum-V-1-bis}
    \| \theta_{k}D^{\sigma-\beta_{k}}&[V \varphi]\|_{2} \\\notag&\le C A
    \Big(\frac{B}{\epsilon}\Big)^{j-k} \sum_{m=0}^{j-k} \binom{j-k}{m}
    \frac{m!(j-k-m)^{j-k-m}}{(j-k)^{j-k}} \Big(\frac{A}{B}\Big)^{m}\,.
  \end{align}
  As in \cite[(62)]{ours}, this implies (choosing \(B>2A\)), that, for
  any \(k\in\{0,\ldots,j-1\}\),
  \begin{align}\label{eq:first-sum-V-1}
    \| \theta_{k}D^{\sigma-\beta_{k}}[V \varphi]\|_{2} &\le C A
    \Big(\frac{B}{\epsilon}\Big)^{j-k} \sum_{m=0}^{j-k}
    \Big(\frac{A}{B}\Big)^{m} \le 2C A
    \Big(\frac{B}{\epsilon}\Big)^{j-k}\,.
  \end{align}
  Note that, by \eqref{f4}, the same estimate holds true if \(k=j\).

  So, from \eqref{f3}, \eqref{eq:est-smooth-factor-BIS},
  \eqref{eq:first-sum-V-1}, the fact that \(\epsilon\le1\) (since
  \(\epsilon(j+1)\le R/2\le1/2\)), and choosing \(B>4n+4,
  B>12A\widetilde{C}_{n,s}(m)\), it follows that (also for
  \(\widetilde{E}_{1/2,0}({\bf p})^{-1}\))
  \begin{align}\label{eq:final-first-sum-V}\notag
    \Big\|\sum_{k=0}^{j} \Phi D_{\nu} E_{s,m}(\mathbf{p})^{-1}
   & D^{\beta_{k}} \chi_{k} D^{\sigma-\beta_{k}} [V \varphi]\Big\|_{2}
    \\\notag&\le 2CA \widetilde{C}_{n,s}(m)\Big(\frac{B}{\epsilon}\Big)^{j}
    \sum_{k=0}^{j} \Big(\frac{2n+2}{B}\Big)^{k}
       \\&\le C(4 A
    \widetilde{C}_{n,s}(m))\Big(\frac{B}{\epsilon}\Big)^{j} \le
    \frac{C}{3}\Big(\frac{B}{\epsilon}\Big)^{j+1}\,. 
  \end{align}

\ %

\noindent{\bf The second sum in \eqref{f2}.}
The second sum in \eqref{f2} is the first one with
  \(j\) replaced by \(j-1\) and \(\chi_k\) replaced by
  \([\eta_k,D^{\mu_k}]=-D^{\mu_k}\eta_k\).  
Hence, using that \(\epsilon\le1\),
  the choice of \(B\) above, and choosing \(B\ge C_{*}\) (see
  \eqref{eq:est-der-loc} for \(C_{*}\)), we get that
  \begin{align}\label{eq:second-sum-total-final-V}
    \Big\|&\sum_{k=0}^{j-1} \Phi
    D_{\nu}E_{s,m}(\mathbf{p})^{-1}D^{\beta_{k}}
    [\eta_{k},D^{\mu_{k}}] D^{\sigma-\beta_{k+1}} [V
    \varphi]\Big\|_{2} \\&\le \frac{C_{*}}{\epsilon} C(4A
    \widetilde{C}_{n,s}(m))\Big(\frac{B}{\epsilon}\Big)^{j-1} \le C
    (4A\widetilde{C}_{n,s}(m)) \Big(\frac{B}{\epsilon}\Big)^{j} \le
    \frac{C}{3}\Big(\frac{B}{\epsilon}\Big)^{j+1}\,.\nonumber
  \end{align}

\ %

\noindent{\bf The last term in \eqref{f2}.}
It remains to study
\begin{align}\label{the-term}
  \Phi D^{\beta} E_{s,m}(\mathbf{p})^{-1}[\eta_{j} V \varphi] \ \text{
    and }\ \Phi D^{\beta}
  \widetilde{E}_{1/2,0}(\mathbf{p})^{-1}[\eta_{j} \widetilde{V}
  \varphi]\,.
\end{align}
Recall that \(\Phi\) is supported in \(\omega_{\epsilon(j+1)}\) and
(see Appendix)
\begin{align}\label{d-bis}
  {\operatorname{dist}}({\operatorname{supp}}\,
  \Phi,{\operatorname{supp}}\, \eta_j)\ge \epsilon(j+1/4)\,.
\end{align}
Recall that \(V=V_1+V_2\in L^{\infty}(\mathbb{R}^{n})+L^{t}(\mathbb{R}^{n})\)
for some \(t\) (see \eqref{cond-L-t-on-V}).  Again, we use
Lemma~\ref{normsmooth-Lp} (twice), this time with
\(\mathfrak{q}=\mathfrak{q}^{*}=2\) (both times), and \(\mathfrak{p}_{1}=2\),
\(\mathfrak{r}_{1}=1\) (for \(V_1\)), and
\(\mathfrak{p}_{2}=\max\{2n/(n+4s),1\}\),
\(\mathfrak{r}_{2}=\min\{n/(n-2s),2\}\) (for \(V_2\)). 
Then (for \(i=1,2\)) \(\mathfrak{p}_{i}^{-1}+\mathfrak{q}^{-1}+\mathfrak{r}_{i}^{-1}=2\),
\(\mathfrak{p}_{i}\in[1,\infty)\), \(\mathfrak{q}>1\), \(\mathfrak{r}_{i}\in
[1,\infty)\), and \(\mathfrak{q}^{-1}+{{\mathfrak{q}}^{*}}^{-1}=1\).
Also, \(|\beta|=j+1\ge2\).
Lemma~\ref{normsmooth-Lp} therefore gives that, for \(i=1,2\),
\begin{align}\label{est-outside}
  \|\Phi D^{\beta} E_{s,m}(\mathbf{p})^{-1}[\eta_{j} V_{i}\varphi]\|_{2}
  &\le \|\Phi D^{\beta}
  E_{s,m}(\mathbf{p})^{-1}\eta_{j}\|_{\mathcal{B}_{\mathfrak{p}_{i},2}}
  \|V_{i}\varphi\|_{\mathfrak{p}_{i}} \\&\le c_{n,s,\mathfrak{r}_{i}}
  \,\beta!\,\Big(\frac{2n+2}{\epsilon(j+1/4)}\Big)^{|\beta|-2s+n(1-1/\mathfrak{r}_{i})}
  \|V_{i}\varphi\|_{\mathfrak{p}_{i}}\,.\notag
\end{align}
As before, we used that \(\|\Phi\|_{\infty}=\|\eta_{j}\|_{\infty}=1\).
The same estimate holds for
\(\widetilde{E}_{1/2,0}(\mathbf{p})^{-1}\).  Note that
\begin{align}\label{est:beta-fak}\notag
  \beta!&\Big(\frac{2n+2}{j+1/4}\Big)^{|\beta|}  \\&\le
  \frac{(8n+8)^{|\beta|}|\beta|!}{(j+1)^{|\beta|}}
  =\frac{(8n+8)^{|\beta|}(j+1)!}{(j+1)^{j+1}}\le (8n+8)^{|\beta|}\,.
\end{align}
Since \(\epsilon(j+1)\le R/2<1\) and \(n(1/\mathfrak{r}_{i}-1)+2s\ge0\)
(for both \(i=1\) and \(i=2\)), it follows that
\((\epsilon(j+1/4))^{n(1/\mathfrak{r}_{i}-1)+2s}\le 1\).  
Therefore (also for \(\widetilde{E}_{1/2,0}(\mathbf{p})^{-1}\))
\begin{align}\label{sec-est-bis-4}
  \|\Phi D^{\beta} E_{s,m}(\mathbf{p})^{-1}\eta_{j} V_{i}\|_{2} \le
  \overline{c}_{n,s, \mathfrak{r}_{i}}
  \Big(\frac{8n+8}{\epsilon}\Big)^{|\beta|}
  \|V_{i}\varphi\|_{\mathfrak{p}_{i}}\,,i=1,2\,.
\end{align}

It remains to note that, using the stated conditions on \(t\) (see
\eqref{cond-L-t-on-V}), Sobolev embedding (for \(\varphi\in
H^{2s}(\mathbb{R}^{n})\)), and H{\"o}lder's inequality, one has (in
all cases), \(\|V_{i}\varphi\|_{\mathfrak{p}_{i}}<\infty\), for the stated
choices of \(\mathfrak{p}_{i}\), \(i=1,2\). Hence, choosing \(B>8n+8\)
and \(C\ge 6\max_{i=1,2}\{\overline{c}_{n,s,\mathfrak{r}_{i}}\|V_{i}\varphi\|_{\mathfrak{p}_{i}}\}\)
(recall that \(|\beta|=j+1\)),
\begin{align}\label{sec-est-bis-5}
  \|\Phi D^{\beta} E_{s,m}(\mathbf{p})^{-1}\eta_{j} V\|_{2}
  \le\frac{C}{3}\Big(\frac{B}{\epsilon}\Big)^{j+1}\,.
\end{align}
The same holds for \(\widetilde{E}_{1/2,0}({\bf p})^{-1}\).

The estimate \eqref{est-V} now follows from \eqref{f2} and the
estimates \eqref{eq:final-first-sum-V},
\eqref{eq:second-sum-total-final-V}, and~\eqref{sec-est-bis-4}.  
 \quad\ 
\end{proof}
\section*{Appendix A.}
Recall (see \eqref{def:Phi}) that we have chosen a function $\Phi$
(depending on \(j\)) satisfying
\begin{equation}\label{def:Phi-BIS}
  \Phi \in C^{\infty}_{0}(\omega_{\epsilon (j+3/4)})\,,\quad
  0\le\Phi\le1\,,\quad
  \mbox{
    with }\; \Phi \equiv 1 \; \mbox{ on }\; \omega_{\epsilon(j+1)}\,. 
\end{equation}

For $j\in\mathbb{N}$ we choose functions $\{\chi_{k}\}_{k=0}^{j}$, and
$\{\eta_{k}\}_{k=0}^{j}$ (all depending on \(j\)) with the following
properties (see \cite[Figures 1 and 2]{ours}).
The functions $\{\chi_{k}\}_{k=0}^{j}$ are such that
\begin{align*}
  \chi_{0} \in C^{\infty}_{0}(\omega_{\epsilon (j+1/4)}) \ \;\mbox{
    with }\; \ \, \chi_{0} \equiv 1 \; \ \ \,\mbox{ on }\
  \omega_{\epsilon(j+1/2)}\,,
\end{align*}
and, for $k= 1, \dots, j$,
\begin{align*}
  \chi_{k} &\in C^{\infty}_{0}(\omega_{\epsilon (j-k+1/4)}) 
  \\&\quad\text{with}
  \quad\begin{cases}
    \chi_{k} \equiv 1 & \text{on $\omega_{\epsilon(j-k+1/2)}\setminus
      \omega_{\epsilon(j-k+1+1/4)}\,,$}
      \\
      \chi_{k} \equiv 0 & \text{on $\mathbb{R}^n \setminus (
      \omega_{\epsilon(j-k+1/4)}\setminus
      \omega_{\epsilon(j-k+1+1/2)})\,.$}
    \end{cases}
\end{align*}
Finally, the functions $\{\eta_{k}\}_{k=0}^{j}$ are such that for $k=
0, \dots, j$,
\begin{equation*}
  \eta_{k} \in C^{\infty}(\mathbb{R}^n) 
  \quad\text{with} \quad
  \begin{cases}
    \eta_{k} \equiv 1 & \text{on $\mathbb{R}^n \setminus
      \omega_{\epsilon(j-k+1/4)}\,,$}
    \\
    \eta_{k} \equiv 0  & \text{on $
    \omega_{\epsilon(j-k+1/2)}\,.$}
  \end{cases}
\end{equation*}
Moreover we ask that
\begin{equation}\label{chieta}
  \begin{array}{lll}
    \chi_{0}+\eta_{0} \equiv 1 
    & \mbox{ on } \mathbb{R}^n,\\ 
    \chi_{k}+\eta_{k} \equiv 1 
    & \mbox{ on } \mathbb{R}^n \setminus
    \omega_{\epsilon(j-k+1+1/4)}\ \mbox{ for }\ k=1,\dots,j\,,\\ 
    \eta_{k} \equiv \chi_{k+1}+\eta_{k+1} 
    & \mbox{ on }  \mathbb{R}^n \ \mbox{
      for }\ k=0,\ldots, j-1\,. 
  \end{array}
\end{equation}
Lastly, we choose these localization functions such that, for a
constant $C_{*}>0$ (independent of \(\epsilon, k, j, \beta\)) and for
all $\beta \in \mathbb{N}_{0}^n$ with $|\beta|=1$, we have that
\begin{equation}\label{eq:est-der-loc}
  |D^{\beta}\chi_{k}(\mathbf{x})| \leq \frac{C_{*}}{\epsilon}
  \quad\text{and}\quad  |D^{\beta}\eta_{k}(\mathbf{x})| \leq
  \frac{C_{*}}{\epsilon}\,, 
\end{equation}
for $k=0, \dots,j$, and all \(\mathbf{x}\in\mathbb{R}^n\).

The next lemma shows how to use these localization functions.
\begin{lemma}\label{Edgardo}
  For $j\in\mathbb{N}$ fixed, choose functions
  $\{\chi_{k}\}_{k=0}^{j}$, and $\{\eta_{k}\}_{k=0}^{j}$ as above, and
  let $\sigma \in \mathbb{N}_{0}^n$ with $|\sigma|=j$. For $\ell \in
  \mathbb{N}$ with $\ell \leq j$, choose multiindices
  $\{\beta_{k}\}_{k=0}^{\ell}\subset \mathbb{N}_{0}^{n}$ such that:
  \begin{equation*}
    |\beta_{k}|=k \qquad\text{for }k=0, \dots, \ell\; ,\;  \beta_{k-1}<\beta_{k}
    \; \mbox{ for } \; k=1, \dots, \ell, \text{ and }  \beta_{\ell}\leq
    \sigma\,. 
  \end{equation*}

  Then for all $g \in \mathcal{S}'(\mathbb{R}^n)$,
  \begin{align}\label{eq:form-localization}\notag
    D^{\sigma}g &= \sum_{k=0}^{\ell} D^{\beta_{k}} \chi_{k}
    D^{\sigma-\beta_{k}} g \\&{\ }+ \sum_{k=0}^{\ell-1}
    D^{\beta_{k}}[\eta_{k},D^{\mu_{k}}] D^{\sigma-\beta_{k+1}} g +
    D^{\beta_{\ell}}\eta_{\ell} D^{\sigma-\beta_{\ell}} g\,,
  \end{align}
  with $\mu_{k}=\beta_{k+1}-\beta_{k}$ for $k=0, \dots, \ell-1$
  (hence, \(|\mu_k|=1\)).
\end{lemma}
For a proof, see \cite[Lemma~B.1]{ours}.

For \(\mathfrak{p},\mathfrak{q}\in[1,\infty]\), denote by
\(\|\cdot\|_{\mathcal{B}_{\mathfrak{p},\mathfrak{q}}}\) the operator
norm on bounded operators from \(L^{\mathfrak{p}}(\mathbb{R}^{n})\) to
\(L^{\mathfrak{q}}(\mathbb{R}^{n})\).
\begin{lemma}\label{normsmooth-Lp}
  Let \(n\ge1\), \(\beta\in\mathbb{N}_{0}^{n}\) with \(|\beta|\ge2\). For \(s\in(0,1)\),
  \(m>0\), let \(E_{s,m}({\bf 
    p})^{-1}=(-\Delta+m^2)^{-s}\). 
  For all
  \(\mathfrak{p},\mathfrak{r}\in[1,\infty)\),
  \(\mathfrak{q}\in(1,\infty)\), with
  \(\mathfrak{p}^{-1}+\mathfrak{q}^{-1}+\mathfrak{r}^{-1}=2\), 
and all $\Phi, \chi \in
  C^{\infty}(\mathbb{R}^n)\cap L^{\infty}(\mathbb{R}^n)$ with
  \begin{align}\label{eq:dist-supports}
    {\operatorname{dist}}({\operatorname{supp}}(\chi),
    {\operatorname{supp}}(\Phi)) \geq d\,,
  \end{align}
  the operator \(\Phi E_{s,m}(\mathbf{p})^{-1}D^{\beta}\chi\) is
  bounded from \(L^{\mathfrak{p}}(\mathbb{R}^n)\) to
  \((L^{\mathfrak{q}}(\mathbb{R}^n))'=
  L^{\mathfrak{q}^{*}}(\mathbb{R}^n)\) (with
  \(\mathfrak{q}^{-1}+{\mathfrak{q}^{*}}^{-1}=1\)), and
  \begin{align}\label{eq:smoothing-est-Lp}\notag
    \|\Phi
    E_{s,m}(\mathbf{p})^{-1}&D^{\beta}\chi\|_{\mathcal{B}_{\mathfrak{p},\mathfrak{q}^{*}}}
     \\&\le c_{n,s,\mathfrak{r}}\,
    \beta!\,\Big(\frac{2n+2}{d}\Big)^{|\beta|-2s+n(1-1/\mathfrak{r})}
    \|\Phi\|_{\infty} \|\chi\|_{\infty}\,.
  \end{align}

  The operator \(\widetilde{E}_{1/2,0}({\bf
    p})^{-1}:=\big((-\Delta)^{1/2}+1\big)^{-1}\) satisfies the estimates
  \eqref{eq:smoothing-est-Lp}  
with \(s=1/2\).
\end{lemma}
\begin{proof}
The proof in \cite[Lemma~C.2]{ours} (for \(n=3, m>0, s=1/2\)) uses the
explicit expression for the Green's function for \(n=3\)
\cite[(IX.30)]{RS2}. We give a modified proof, which holds for all
\(n\ge1\), and \(m>0,s\in(0,1)\), as well as \(s=1/2, m=0\), referring to
\cite{ours} for further details. 

We start with \(m>0\), \(s\in(0,1)\). As a substitute for \cite[(C.5)]{ours} we use
\begin{align}\label{integral-formula-power-s}
  \frac{1}{x^s}=c_s\int_0^{\infty}\frac{1}{x+t}\,\frac{dt}{t^s}\ , \ x>0\,.
\end{align}
Denote
\begin{align}\label{green}
   (-\Delta+\lambda)^{-1}({\bf x},{\bf y})&=G_{n}^{\lambda}({\bf
     x}-{\bf y})\ , \quad\lambda>0\,,\  {\bf x}\,,{\bf y}\in\mathbb{R}^n\,, {\bf
     x}\neq{\bf y}\,.
\end{align}
Proceeding as in the proof of \cite[Lemma~C.2]{ours}, we need to
estimate \(\|H\|_{\mathfrak{r}}\), where
\begin{align}\label{formula-H}
  H({\bf z})=c_s
   {\1}_{\{|\,\cdot\,|\ge d\}}({\bf z})(-1)^{|\beta|}\int_0^{\infty}\big(D_{\bf
    z}^{\beta}G_{n}^{m^2+t}\big)({\bf z})\,\frac{dt}{t^s}\,.
\end{align}

For \(n\ge2\), we use (see e.\
g. \cite[(A.11)]{scott}) that 
\begin{align}\label{green-n-big}
   G_{n}^{\lambda}({\bf z})&=\int_0^{\infty}(4\pi u)^{-n/2}{\rm
     e}^{-\lambda u-|{\bf z}|^2/4u}\,du\ , \ {\bf z}\in \mathbb{R}^{n}\setminus\{0\}\,.
\end{align}
Modifying the proof of \cite[Lemma~C.3]{ours} appropriately (choose
\(r=|{\bf x}|/k\), \(\epsilon=1/k, k=2n+2\)), one has that 
\begin{align}\label{estimate-der}
  \big|(D_{\bf x}^{\beta}{\rm e}^{-|{\bf x}|^{2}})({\bf x})\big|\le \beta!\,
  \Big(\frac{2n+2}{|{\bf x}|}\Big)^{|\beta|}{\rm e}^{-|{\bf x}|^2/2}\,,\
  {\bf x}\in\mathbb{R}^{n}\setminus\{0\}\,,\,\beta\in\mathbb{N}_{0}^{n}\,.
\end{align}
Using this, and noting that 
\begin{align}\label{argument1}
  \big(D_{{\bf z}}^{\beta} {\rm e}^{-|{\bf z}|^2/4u}\big)({\bf z})
  = (2\sqrt{u})^{-|\beta|}(D^{\beta}f)({\bf z}/2\sqrt{u})\ ,
  \ f({\bf x})={\rm e}^{-|{\bf x}|^2}\,,
\end{align}
we get (since \({\rm e}^{-m^2u}\le1\)) the estimate 
\begin{align}\label{est-H}\nonumber
  \int_0^{\infty}\Big|\big(&D_{\bf
    z}^{\beta}G_{n}^{m^2+t}\big)({\bf z})\Big|\,\frac{dt}{t^s}
  \\&\le \Gamma(1-s)\beta!\,\Big(\frac{2n+2}{|{\bf
      z}|}\Big)^{|\beta|}\int_0^{\infty}u^{s-1}(4\pi u)^{-n/2}{\rm
     e}^{-|{\bf z}|^2/8u}\,du\,.
\end{align}
Making the change of variables \(v=|{\bf z}|^2/8u\), this implies that
\begin{align}\label{est-H2}
  \int_0^{\infty}\Big|\big(D_{\bf
    z}^{\beta}G_{n}^{m^2+t}\big)({\bf z})\Big|\,\frac{dt}{t^s}
  \le C_{n,s}\,\beta!\,\frac{1}{|{\bf z}|^{n-2s}}\Big(\frac{2n+2}{|{\bf
      z}|}\Big)^{|\beta|}\,.
\end{align}
Upon integration, this gives the estimate
\begin{align}\label{est-norm-H}\nonumber
  \|H\|_{\mathfrak{r}}&\le
  C_{n,s,\mathfrak{r}}\,\beta!\,(2n+2)^{|\beta|}
   \Big(\int_d^{\infty}|{\bf z}|^{-\mathfrak{r}(|\beta|+n-2s)}
  |{\bf z}|^{n-1}\,d|{\bf z}|\Big)^{1/\mathfrak{r}}
  \\& 
  \le\widetilde{C}_{n,s,\mathfrak{r}}\beta!\,
  \Big(\frac{2n+2}{d}\Big)^{|\beta|} 
  d^{\,n/\mathfrak{r}-(n-2s)}\big(\mathfrak{r}(|\beta|+n-2s)-n\big)^{-1/\mathfrak{r}}\,.
\end{align}
The integral is finite since, using \(|\beta|\ge2\), \(s\in(0,1)\),
\(\mathfrak{r}\ge1\), we
have
\(|\beta|-2s+n(1-1/\mathfrak{r})\ge 2-2s>0\). This also implies that
\begin{align}\label{stupid-factor}
  \big(\mathfrak{r}(|\beta|+n-2s)-n\big)^{-1/\mathfrak{r}}
  \le \frac{1}{\mathfrak{r}^{1/\mathfrak{r}}}\frac{1}{(2-2s)^{1/\mathfrak{r}}}\,,
\end{align}
and so
\begin{align}\label{final-H-n}\notag
  \|H\|_{\mathfrak{r}}&\le\tilde{c}_{n,s,\mathfrak{r}}\,\beta!\,
  \Big(\frac{2n+2}{d}\Big)^{|\beta|} d^{\,n/\mathfrak{r}-(n-2s)}
  \\&=c_{n,s,\mathfrak{r}}\,\beta!\,\Big(\frac{2n+2}{d}\Big)^{|\beta|-2s+n(1-1/\mathfrak{r})}\,.
\end{align}

For \(n=1\), \(G_1^{\lambda}(z)={\rm
  e}^{-\sqrt{\lambda}|z|}/2\sqrt{\lambda}\), \(z\in\mathbb{R}\),
\(\lambda>0\) 
(use \eqref{green-n-big}), so for \(\beta\in\mathbb{N}\),
\begin{align}\label{H-n=1}
  |H(z)|
    \le \frac{c_s}{2}  {\1}_{\{|\,\cdot\,|\ge
    d\}}(z)\int_0^\infty\big(\sqrt{m^2+t}\big)^{\beta-1}{\rm
    e}^{-\sqrt{m^2+t}\,|z|}\,\frac{dt}{t^s} \,.
\end{align}
Now, by Minkowski's inequality \cite[Theorem 2.4]{LiebLoss}, for
\(\mathfrak{r}\ge1\), 
\begin{align}\label{Minkowski-n=1}\nonumber
  \|H\|_{{\mathfrak{r}}}&
   \le
  \frac{c_s}{2}\int_0^\infty\big(\sqrt{m^2+t}\big)^{\beta-1}\Big(2\int_d^\infty
  {\rm
    e}^{-\mathfrak{r}\sqrt{m^2+t}\,z}\,dz\Big)^{1/\mathfrak{r}}\,\frac{dt}{t^s}
  \nonumber
  \\&=\frac{c_s2^{1/\mathfrak{r}-1}}{\mathfrak{r}^{1/\mathfrak{r}}}
  \int_0^{\infty}\big(\sqrt{m^2+t}\big)^{\beta-1-1/\mathfrak{r}}  
  {\rm e}^{-d\sqrt{m^2+t}}\,\frac{dt}{t^s}\,.
\end{align}
Now (maximise \(x^k{\rm e}^{-\alpha x}\); note that
\(\beta-1-1/\mathfrak{r}\ge0\) when \(\beta\ge2\), \(\mathfrak{r}\ge1\)),
\begin{align}\label{TOS' trick}
  \big(\sqrt{m^2+t}\big)^{\beta-1-1/\mathfrak{r}}{\rm
    e}^{-d\sqrt{m^2+t}/2}
  \le\Big(\frac{2(\beta-1-1/\mathfrak{r})}{{\rm
      e}d}\Big)^{\beta-1-1/\mathfrak{r}}\,. 
\end{align}
Hence,
\begin{align}\label{H-n=1-second}\nonumber
  \|H\|_{{\mathfrak{r}}}&\le c_{s}
  \frac{2^{1/\mathfrak{r}-1}}{\mathfrak{r}^{1/\mathfrak{r}}}
  \Big(\frac{2(\beta-1-1/\mathfrak{r})}{{\rm
      e}d}\Big)^{\beta-1-1/\mathfrak{r}}
  \int_0^\infty{\rm e}^{-d\sqrt{t}/2}\,\frac{dt}{t^s}
  \\&=
  \frac{\overline{c}_{s,r}}{d^{2-2s}}
  \Big(\frac{2(\beta-1-1/\mathfrak{r})}{{\rm
      e}d}\Big)^{\beta-1-1/\mathfrak{r}}\,.
\end{align}
It remains to note that, using 
\cite[(A.7)]{ours},  
we have that, for some \(\vartheta=\vartheta(\beta) \in (0,1)\), 
\begin{align}\label{beta-beta-to-factorial}\notag
  (\beta-1-1/\mathfrak{r})^{\beta-1-1/\mathfrak{r}}
  &\le(\beta-1)^{\beta-1/2}\\&=\frac{(\beta-1)!}{\sqrt{2\pi}}\,{\rm
    e}^{\beta-\vartheta/12(\beta-1)}
  \le \beta!\,{\rm e}^{\beta}\,.
\end{align}
Hence,
\begin{align}\label{H-final-n=1}
  \|H\|_{\mathfrak{r}}\le
  \tilde{c}_{1,s,\mathfrak{r}}\,\beta!\,\Big(\frac{4}{d}\Big)^{\beta}d^{1/\mathfrak{r}-(1-2s)}
  =c_{1,s,\mathfrak{r}}\,\beta!\,\Big(\frac{4}{d}\Big)^{\beta-2s+(1-1/\mathfrak{r})}\,.
\end{align}

For \(\widetilde{E}_{1/2,0}({\bf p})^{-1}\), one uses, for \(n\ge1\),
\begin{align}\label{eq:kernel-mass-zero}\notag
  \widetilde{E}_{1/2,0}&({\bf p})^{-1}({\bf x},{\bf y})
  \\&= \Gamma\Big(\frac{n+1}{2}\Big)\pi^{-(n+1)/2}
  \int_0^{\infty}\frac{{\rm e}^{-t|{\bf x}-{\bf
        y}|}}{|{\bf x}-{\bf y}|^{n-1}}\,\frac{t\,dt}{(t^2+1)^{(n+1)/2}}\,,
\end{align}
which follows from \((x+1)^{-1}=\int_0^{\infty}{\rm
  e}^{-tx}e^{-t}\,dt\), and the explicit expression for the heat
kernel in \(\mathbb{R}^{n}\) of \((-\Delta)^{1/2}\) (see
\cite[Section~7.11(10)]{LiebLoss}). 

For \(n\ge2\), one also needs the estimate
\begin{align}\label{eq:est-der}\nonumber
  \Big|D_{\mathbf{x}}^{\beta}\frac{{\rm
      e}^{-t|\mathbf{x}|}}{|\mathbf{x}|^{n-1}}\Big| 
  \le \beta!\Big(\frac{\sqrt{2}}{|\mathbf{x}|}\Big)^{n-1}
  &\Big(\frac{2n+2}{|\mathbf{x}|}\Big)^{|\beta|}{\rm
    e}^{-t|\mathbf{x}|/2}\ \\
  &\text{ for all } t>0\,, \mathbf{x}\in
  \mathbb{R}^n\setminus\{0\}\,, \beta\in\mathbb{N}_{0}^{n}\,,
\end{align}
which follows as for \cite[Lemma~C.3 (C.9)]{ours} (with the
modifications mentioned above).\quad\ {}
Then, with \(c_n=\Gamma(n+1/2)\pi^{-(n+1)/2}\), using \({\rm
    e}^{-t|\mathbf{z}|/2}\le1\),
\begin{align}\label{H massless}
  |H({\bf z})|&\le
  c_n {\1}_{\{|\,\cdot\,|\ge
    d\}}({\bf z})\int_0^\infty \Big(D_{\bf z}^{\beta}\frac{{\rm e}^{-t|{\bf
        z}|}}{|{\bf z}|^{n-1}}\Big)\,\frac{t\,dt}{(t^2+1)^{(n+1)/2}}
  \\&\le c_n  {\1}_{\{|\,\cdot\,|\ge
    d\}}({\bf z})\,\beta!\,\Big(\frac{\sqrt{2}}{|{\bf
      z}|}\Big)^{n-1}\Big(\frac{2n+2}{|{\bf z}|}\Big)^{|\beta|}
   \int_0^{\infty}\frac{t\,dt}{(t^2+1)^{(n+1)/2}}\,.\nonumber
\end{align}
The integral is finite since \(n\ge2\). Now proceed as after
\eqref{est-H2}. 

For \(n=1\), \(\beta\in\mathbb{N}\), again using
\eqref{eq:kernel-mass-zero}, 
\begin{align}\label{H massless n=1}
  |H(z)|\le  c_1 
  {\1}_{\{|\,\cdot\,|\ge d\}}(z)
  \int_0^\infty t^{\beta}{\rm e}^{-t|z|}\,\frac{t\,dt}{t^2+1}\,.
\end{align}
By Minkowski's inequality, for \(\beta\ge 2\), \(\mathfrak{r}\ge1\),
 \begin{align}\label{H-norm-massless n=1}\notag
   \|H\|_{\mathfrak{r}}&\le
   c_1\Big(\frac{2}{\mathfrak{r}}\Big)^{1/\mathfrak{r}}
   \int_0^\infty t^{\beta-1/\mathfrak{r}}{\rm
     e}^{-dt}\,\frac{t\,dt}{t^2+1}
   \\&\le \overline{c}_{1,n,\mathfrak{r}}\Big(\frac{1}{d}\Big)^{\beta}
   d^{1/\mathfrak{r}} \int_0^\infty t^{\beta-1/\mathfrak{r}}{\rm e}^{-t}\,\frac{dt}{t}\,.
 \end{align}
The last integral is finite, since \(\beta\ge2\),
\(\mathfrak{r}\ge1\), and can be
bounded by
 \begin{align}\label{last int bound}\notag
   \int_0^1 t^{\beta-1/\mathfrak{r}}{\rm e}^{-t}\,\frac{dt}{t}+
   \int_1^\infty t^{\beta-1/\mathfrak{r}}{\rm e}^{-t}\,\frac{dt}{t}
   &\le \Gamma(\beta-1)+\Gamma(\beta)\\&\le 2(\beta-1)!\le\beta!\,.
 \end{align}

The rest of the proof is (in all cases) as in \cite{ours}.
\end{proof}


\begin{thebibliography}{1}

\bibitem{ours}
A.~Dall'Acqua, S.~Fournais, T.~{\O}. S{\o}rensen and E.~Stockmeyer, {\em Anal.
  PDE} {\bf 5}, 657 (2012).

\bibitem{Hormander}
L.~H{\"o}rmander, {\em Linear partial differential operators}, Die
Grundlehren der mathematischen Wissenschaften, Band 116, Third 
  revised printing, Springer-Verlag New York Inc., New York (1969).

\bibitem{Krantz}
S.~G. Krantz and H.~R. Parks, {\em A primer of real analytic
  functions}, Birkh\"auser Advanced Texts: Basler Lehrb\"ucher, second
edn., Birkh\"auser Boston Inc., Boston, MA (2002).  

\bibitem{RS2}
M.~Reed and B.~Simon, {\em Methods of modern mathematical physics. {II}.
  {F}ourier analysis, self-adjointness}, Academic Press, New York (1975).

\bibitem{scott}
J.~P. Solovej, T.~{\O}. S{\o}rensen and W.~L. Spitzer, {\em Comm. Pure Appl.
  Math.} {\bf 63}, 39 (2010).

\bibitem{LiebLoss}
E.~H. Lieb and M.~Loss, {\em Analysis}, Graduate Studies in Mathematics,
  Vol.~14, second edn., American Mathematical Society, Providence, RI, (2001).

\end{thebibliography}
\end{document}